\documentclass{amsart}
\usepackage[utf8]{inputenc}
\usepackage{latexsym,amssymb,amsmath,amsthm,amscd,graphicx}
\usepackage{amsfonts}
\usepackage{xcolor}
\usepackage[english]{babel}
\usepackage[a4paper,bindingoffset=0.2in,%
            left=1in,right=1in,top=1in,bottom=1in,%
            footskip=.25in]{geometry}
\newtheorem{theorem}{Theorem}[section]
\newtheorem{remark}[theorem]{Remark}

\begin{document}

\title[On geometric constants for (small) Morrey spaces]{On geometric constants for (small) Morrey spaces}

\author[A.~Mu'tazili]{Aqfil Mu'tazili}
\address{Faculty of Mathematics and Natural Sciences, Bandung Institute of Technology, Bandung 40132, Indonesia}
\email{aqfil.mt@gmail.com}

\author[H.~Gunawan]{Hendra Gunawan}
\address{Faculty of Mathematics and Natural Sciences, Bandung Institute of Technology, Bandung 40132, Indonesia}
\email{hgunawan@math.itb.ac.id}



\begin{abstract}
    In this article, we compute Von Neumann--Jordan constant, James constant, and Dunkl--Williams constant for small
    Morrey spaces. Our approach can also be seen as an alternative way in computing the three constants for the
    (classical) Morrey spaces. In addition, we prove constructively that Morrey spaces are not uniformly non-octahedral.
\end{abstract}

\maketitle

\section{Introduction}

Let $1\leq p\leq q<\infty$. The (classical) {\em Morrey space} $\mathcal{M}^p_q=\mathcal{M}^p_q(\mathbb{R}^d)$
is the set of all measurable functions $f$ such that
\begin{equation*}
    \|f\|_{\mathcal{M}^p_q} := \sup_{{a\in \mathbb{R}^d},{R>0}}|B(a,R)|^{\frac{1}{q}-\frac{1}{p}}
    \biggl(\int\limits_{B(a,R)} |f(y)|^p dy\biggr)^{\frac{1}{p}}<\infty,
\end{equation*}
where $|B(a,R)|$ denotes the Lebesgue measure of the open ball $B(a,R)$ in $\mathbb{R}^d$, with center $a$
and radius $R$. Meanwhile, the {\em small Morrey space} $m^p_q=m^p_q(\mathbb{R}^d)$ is defined to be the set of
all measurable functions $f$ such that
\begin{equation*}
\|f\|_{m^p_q} := \sup_{{a\in \mathbb{R}^d},{R\in(0,1)}}|B(a,R)|^{\frac{1}{q}-\frac{1}{p}}
\biggl(\int\limits_{B(a,R)} |f(y)|^p dy\biggr)^{\frac{1}{p}}<\infty.
\end{equation*}
Morrey spaces and small Morrey spaces are Banach spaces \cite{sawano}, and for each $p$ and $q$ the small
Morrey space properly contains the Morrey space. Note that for $p=q$, the space $m^q_q$ is identical with
the $L^q_{\rm uloc}$ space \cite{sawano}.

In \cite{2018arXiv181000963G}, three geometric constants, namely {\em Von Neumann--Jordan constant}, {\em James constant},
and {\em Dunkl--Williams constant}, have been computed for Morrey spaces.
For a Banach space $(X,\|\cdot\|_X)$ in general, the three constants are defined by
\begin{equation*}
    C_{\rm NJ}(X) := \sup\biggl\{\frac{\|x+y\|_X^2+\|x-y\|_X^2}{2(\|x\|_X^2+\|y\|_X^2)} : x,y \in X\setminus \{0\} \biggr\},
\end{equation*}
\begin{equation*}
    C_{\rm J}(X) := \sup\bigl\{\min\{\|x+y\|_X,\|x-y\|_X\} : x,y\in X, \|x\|_X=\|y\|_X=1 \bigr\},
\end{equation*}
and
\begin{equation*}
    C_{\rm DW}(X) := \sup\biggl\{\frac{\|x\|_X+\|y\|_X}{\|x-y\|_X}\biggl\|\frac{x}{\|x\|_X}-\frac{y}{\|y\|_X}\biggr\|_X :
    x,y \in X\backslash \{0\}, x\neq y \biggr\}
\end{equation*}
respectively. These constants measure {\em uniformly non-squareness} of $X$ \cite{james, kato}.

Here are some known facts about these constants:
\begin{itemize}
    \item $1\leq C_{\rm NJ}(X)\leq 2$ and $C_{\rm NJ}(X)=1$ if and only if $X$ is a Hilbert space \cite{jordan1935inner}.
    \item $\sqrt{2}\leq C_{\rm J}(X)\leq 2$ and $C_{\rm J}(X)=\sqrt{2}$ if $X$ is a Hilbert space \cite{gao1990geometry}.
    \item $2\leq C_{\rm DW}(X)\leq 4$ and $C_{\rm DW}(X)=2$ if and only if $X$ is a Hilbert space \cite{jimenez2008dunkl}.
    \item For $1\le p\le \infty$, $C_{\rm NJ}(L^p)=\max\{2^{\frac{2}{p}-1},2^{1-\frac{2}{p}}\}$ \cite{10.2307/1968512}
    and $C_{\rm J}(L^p)=\max\{2^{\frac{1}{p}},2^{1-\frac{1}{p}}\}$ \cite{gao1990geometry}. Here $L^p=L^p(\mathbb{R}^d)$
    denotes the space of $p$-th power integrable functions on $\mathbb{R}^d$.
\end{itemize}

For the case of Morrey spaces, we know the following results.

\begin{theorem}{\rm \cite{2018arXiv181000963G}}\label{Morrey}
For $1\le p<q<\infty$, we have $C_{\rm NJ}(\mathcal{M}^p_q)=C_{\rm J}(\mathcal{M}^p_q)=2$
and $C_{\rm DW}(\mathcal{M}^p_q)=4$.
\end{theorem}

In this article, we shall compute the three constants for small Morrey spaces. We prove that the constants for
small Morrey spaces are the same as those for Morrey spaces. Our proof may be considered as an alternative
proof of the above theorem. In addition, we also prove that the (classical) Morrey spaces
are not {\em uniformly non-octahedral}.

\section{Main Results}

Our main results are presented in the following theorem. We prove the theorem constructively, using different
functions than those used in \cite{2018arXiv181000963G}.

\begin{theorem}
Let $1\leq p< q<\infty$. Then $C_{\rm NJ}(m^p_q)=C_{\rm J}(m^p_q)=2$ and $C_{\rm DW}(m^p_q)=4$.
\end{theorem}

\begin{proof}
Define $f(x):=\chi_{(0,1)}(|x|)|x|^{-\frac{d}{q}}$, $g(x):=\chi_{(0,\epsilon)}(|x|)f(x)$,
$h(x):=\chi_{(\epsilon,1)}(|x|)f(x)$, $k(x):=g(x)-h(x)$ and $l(x):=(1+\delta)g(x)+(1-\delta)h(x)$
with $0<\epsilon,\delta<1$.
Here $g, h$ and $k$ depend on $\epsilon$, while $l$ depends on $\epsilon$ and $\delta$.
Noting that all functions are radial functions, $f,g, l$ are radially decreasing, $0\le h\le f$,
and $|k|=f$, we have
\begin{equation*}
    \|f\|_{m^p_q}=\sup_{R\in(0,1)} |B(0,R)|^{\frac1q-\frac1p} \biggl(\int\limits_{B(0,R)}|y|^{-\frac{dp}{q}}dy
    \biggr)^{\frac{1}{p}}=\Bigl(\frac{C_d}{d}\Bigr)^{\frac{1}{q}} \Bigl(1-\frac{p}{q}\Bigr)^{-\frac{1}{p}},
\end{equation*}
\begin{equation*}
     \begin{split}
    \|g\|_{m^p_q} &= \sup_{R\in(0,1)}|B(0,R)|^{\frac{1}{q}-\frac{1}{p}}\biggl(\int\limits_{B(0,R)}
    |\chi_{(0,\epsilon)}(|y|)f(y)|^p dy\biggr)^{\frac{1}{p}}\\
    &= \sup_{{R\in(0,\epsilon)}} |B(0,R)|^{\frac1q-\frac1p} \biggl(\int\limits_{B(0,R)}
    |y|^{-\frac{dp}{q}}dy\biggr)^{\frac{1}{p}}\\
    &= \Bigl(\frac{C_d}{d}\Bigr)^{\frac{1}{q}}\Bigl(1-\frac{p}{q}\Bigr)^{-\frac{1}{p}}\\
    &= \|f\|_{m^p_q},
    \end{split}
\end{equation*}
\begin{equation*}
    \begin{split}
    \|h\|_{m^p_q} &\ge \sup_{{R\in(0,1)}}|B(0,R)|^{\frac{1}{q}-\frac{1}{p}}\biggl(\int\limits_{B(0,R)}
    |\chi_{(\epsilon,1)}(|y|)f(y)|^p dy\biggr)^{\frac{1}{p}}\\
    &= \sup_{{R\in(\epsilon,1)}} |B(0,R)|^{\frac{1}{q}-\frac{1}{p}}\biggl(\int\limits_{B(0,R)\backslash B(0,\epsilon)}
    	|y|^{-\frac{dp}{q}} dy\biggr)^{\frac{1}{p}}\\
    &= \sup_{R\in (\epsilon,1)}\Bigl(\frac{C_d}{d}\Bigr)^{\frac{1}{q}-\frac{1}{p}} R^{\frac{d}{q}-\frac{d}{p}}
    \biggl(C_d\int\limits_\epsilon^R r^{-\frac{dp}{q}+d-1}dr\biggr)^{\frac{1}{p}}\\
    &= \sup_{R\in (\epsilon,1)} \Bigl(\frac{C_d}{d}\Bigr)^{\frac{1}{q}}\Bigl(1-\frac{p}{q}\Bigr)^{-\frac{1}{p}}
    (1-R^{\frac{dp}{q}-d}\epsilon^{-\frac{dp}{q}+d})^{\frac{1}{p}}\\
    &= (1-\epsilon^{-\frac{dp}{q}+d})^{\frac{1}{p}}\|f\|_{m^p_q},
\end{split}
\end{equation*}
and
\begin{equation*}
    \begin{split}
        \|k\|_{m^p_q} &= \sup_{{R\in(0,1)}}|B(0,R)|^{\frac{1}{q}-\frac{1}{p}}\biggl(\int\limits_{B(0,R)}
        |[\chi_{(0,\epsilon)}(|y|)-\chi_{(\epsilon,1)}(|y|)]f(y)|^p
        dy\biggr)^{\frac{1}{p}}\\
        &= \sup_{{R\in(0,1)}}|B(0,R)|^{\frac{1}{q}-\frac{1}{p}}\biggl(\int\limits_{B(0,R)} |f(y)|^p dy
        \biggr)^{\frac{1}{p}}\\
        &= \|f\|_{m^p_q},
    \end{split}
\end{equation*}
where $C_d$ denotes the `area' of the unit sphere in $\mathbb{R}^d$. We observe that
\begin{equation*}
    \begin{split}
        C_{\rm NJ}(m^p_q)
        &\ge \frac{\|f+k\|_{m^p_q}^2+\|f-k\|_{m^p_q}^2}{2(\|f\|_{m^p_q}^2+\|k\|_{m^p_q}^2)}\\
        &= \frac{\|2g\|_{m^p_q}^2+\|2h\|_{m^p_q}^2}{2(\|f\|_{m^p_q}^2+\|k\|_{m^p_q}^2)}
        = \frac{4\|g\|_{m^p_q}^2+4\|h\|_{m^p_q}^2}{2(\|f\|_{m^p_q}^2+\|k\|_{m^p_q}^2)}\\
        &\ge \frac{2\bigl(\|f\|_{m^p_q}^2+(1-\epsilon^{-\frac{dp}{q}+d})^{\frac{2}{p}}\|f\|_{m^p_q}^2
        \bigr)}{\|f\|_{m^p_q}^2+\|f\|_{m^p_q}^2}\\
        &= 1+(1-\epsilon^{-\frac{dp}{q}+d})^{\frac{2}{p}},\\
    \end{split}
\end{equation*}
and this holds for any $0<\epsilon<1$. As $\epsilon$ can be arbitrarily small, we obtain $C_{\rm NJ}(m^p_q)\geq2$.
Since $C_{\rm NJ}(m^p_q)\leq 2$, we conclude that $C_{\rm NJ}(m^p_q)=2$.

Next, for James constant, we observe that
\begin{equation*}
    \begin{split}
        \min\{\|f+k\|_{m^p_q},\|f-k\|_{m^p_q}\}&=\min\{\|2g\|_{m^p_q},\|2h\|_{m^p_q}\}\\
        &\ge 2 \min\{\|f\|_{m^p_q},(1-\epsilon^{-\frac{dp}{q}+d})^{\frac{1}{p}}\|f\|_{m^p_q}\}\\
        &= 2(1-\epsilon^{-\frac{dp}{q}+d})^{\frac{1}{p}}\|f\|_{m^p_q}.
    \end{split}
\end{equation*}
By dividing both sides by $\|f\|_{m^p_q}$ ($=\|k\|_{m^p_q}$), we get
\begin{equation*}
    C_{\rm J}(m^p_q) \ge \min\biggl\{\biggl\|\frac{f}{\|f\|_{m^p_q}}+\frac{k}{\|k\|_{m^p_q}}\biggr\|_{m^p_q},
    \biggl\|\frac{f}{\|f\|_{m^p_q}}-\frac{k}{\|k\|_{m^p_q}}\biggr\|_{m^p_q}\biggr\}=2(1-\epsilon^{-\frac{dp}{q}
    +d})^{\frac{1}{p}},
\end{equation*}
and this also holds for any $0<\epsilon<1$. Using similar arguments as before, we conclude that $C_{\rm J}(m^p_q)=2$.

Finally, let us compute the Dunkl--Williams constant. Observe that
\begin{equation*}
    \begin{split}
       \|l\|_{m^p_q} &= \sup_{{R\in(0,1)}}|B(0,R)|^{\frac{1}{q}-\frac{1}{p}}\biggl(\int\limits_{B(0,R)}
       |(1+\delta)g(y)+(1-\delta)h(y)|^p dy\biggr)^{\frac{1}{p}}\\
       &= \sup_{{R\in(0,1)}}|B(0,R)|^{\frac{1}{q}-\frac{1}{p}}\biggl(\int\limits_{B(0,R)}
       |[1+\delta\chi_{(0,\epsilon)}(|y|)-\delta\chi_{(\epsilon,1)}(|y|)]f(y)|^p dy\biggr)^{\frac{1}{p}}\\
       &= \sup_{{R\in(0,1)}}\Bigl(\frac{C_d}{d}\Bigr)^{\frac{1}{q}-\frac1p} R^{\frac{d}{q}-\frac{d}{p}}\biggl(C_d
       \int\limits_0^R [1+\delta\chi_{(0,\epsilon)}(r)-\delta\chi_{(\epsilon,1)}(r)]^pr^{-\frac{dp}{q}+d-1} dr\biggr)^{\frac{1}{p}}\\
       &=\Bigl(\frac{C_d}{d}\Bigr)^{\frac{1}{q}}\Bigl(1-\frac{p}{q}\Bigr)^{-\frac{1}{p}}(1+\delta)\\
       &= (1+\delta)\|f\|_{m^p_q}.
    \end{split}
\end{equation*}
This implies that
\begin{equation*}
    \begin{split}
        C_{\rm DW}(m^p_q) \ge &\ \frac{\|f\|_{m^p_q}+\|l\|_{m^p_q}}{\|f-l\|_{m^p_q}}\biggl\|\frac{f}{\|f\|_{m^p_q}}-
        \frac{l}{\|l\|_{m^p_q}}\biggr\|_{m^p_q}\\
        = &\ \frac{\|f\|_{m^p_q}+(1+\delta)\|f\|_{m^p_q}}{\|f-((1+\delta)g+(1-\delta)h)\|_{m^p_q}}\biggl\|\frac{f}
        {\|f\|_{m^p_q}}-\frac{(1+\delta)g+(1-\delta)h}{(1+\delta)\|f\|_{m^p_q}}\biggr\|_{m^p_q}\\
        = &\ \frac{(2+\delta)\|f\|_{m^p_q}}{\|\delta k\|_{m^p_q}}\biggl\|\frac{2\delta h}{(1+\delta)\|f\|_{m^p_q}}
        \biggr\|_{m^p_q}\\
        \ge &\ \frac{(2+\delta)\|f\|_{m^p_q}}{\delta\|f\|_{m^p_q}}\frac{2\delta
        (1-\epsilon^{-\frac{dp}{q}+d})^{\frac{1}{p}}\|f\|_{m^p_q}}{(1+\delta)\|f\|_{m^p_q}}\\
        = &\ \frac{(4+2\delta)}{1+\delta}(1-\epsilon^{-\frac{dp}{q}+d})^{\frac{1}{p}},\\
    \end{split}
\end{equation*}
and this holds for any $0<\delta,\epsilon<1$. As $\delta$ and $\epsilon$ can be arbitrarily small, we must have
$C_{\rm DW}(m^p_q)\geq 4$.
Using the same arguments as before, we conclude that $C_{\rm DW}(m^p_q)=4$.
\end{proof}

\begin{remark}{\rm
Note that since $f, g, h, k, l \in \mathcal{M}^p_q$ with $\|\diamond\|_{\mathcal{M}^p_q}=\|\diamond\|_{m^p_q}=\|f\|_{m^p_q}$
for $\diamond=f,g,k,l$ and $\|h\|_{\mathcal{M}^p_q}\ge\|h\|_{m^p_q}\ge (1-\epsilon^{-\frac{dp}{q}+d})^{\frac{1}{p}}\|f\|_{m^p_q}$,
we obtain that $C_{\rm NJ}(\mathcal{M}^p_q)=C_{\rm J}(\mathcal{M}^p_q)
=2$ and $C_{\rm DW}(\mathcal{M}^p_q)=4$ for $1\leq p< q<\infty$, as stated in Theorem \ref{Morrey}.}
\end{remark}

\section{Additional Results}

Recall that a Banach space $X$ with a norm $\|\cdot\|_X$ is uniformly non-square if there exists
a $\delta>0$ such that
\[
\min\{\|x+y\|_X,\|x-y\|_X\} \le 2(1-\delta)
\]
for all $x,y\in X$ with $\|x\|_X=\|y\|_X=1$ \cite{james}. The results in the previous section, especially the fact
that $C_{\rm J}(\mathcal{M}^p_q)=2$, implies that the Morrey spaces $\mathcal{M}^p_q$ are not uniformly non-square
(and hence not uniformly convex) for $1\leq p<q<\infty$.

Using seven functions with (approximately) the same Morrey norm, we can also prove that the Morrey spaces
$\mathcal{M}^p_q$ (where $1\le p<q<\infty$) are not
uniformly non-octahedral, that is, there does not exist a $\delta>0$ such that
\[
\min \|F\pm K\pm U\|_{\mathcal{M}^p_q} \le 3(1-\delta)
\]
for every $F,K,U\in \mathcal{M}^p_q$ with $\|F\|_{\mathcal{M}^p_q}=\|K\|_{\mathcal{M}^p_q}=\|U\|_{\mathcal{M}^p_q}=1$.
Here the minimum is taken over all choices of signs in the expression $F\pm K\pm U$.
This follows from the following theorem.

\begin{theorem}\label{Octahedral}
Let $1\leq p< q<\infty$. Then, for every $\delta>0$, there exist $F,K,U\in \mathcal{M}^p_q$ (depending on
$\delta$) with $\|F\|_{\mathcal{M}^p_q}=\|K\|_{\mathcal{M}^p_q}=\|U\|_{\mathcal{M}^p_q}=1$ such that
\[
\|F\pm K\pm U\|_{\mathcal{M}^p_q}> 3(1-\delta)
\]
for all choices of signs.
\end{theorem}

\begin{proof}
Let $f(x):=|x|^{-\frac{d}{q}},\ k(x):=[\chi_{(0,1)}(|x|)-\chi_{(1,\infty)}(|x|)]f(x),\
u(x):=[\chi_{(0,\epsilon)}(|x|)-\chi_{(\epsilon,1)}(|x|)+\chi_{(1,\frac{1}{\epsilon})}(|x|)-
\chi_{(\frac{1}{\epsilon},\infty)}(|x|)]f(x)$, where $0<\epsilon<1$.
Then we have
\begin{align*}
3g &\le |f+k+u| \le 3f,\\
3h &\le |f+k-u| \le 3f,\\
3v &\le |f-k+u| \le 3f,\\
3w &\le |f-k-u| \le 3f,
\end{align*}
where $g(x):=\chi_{(0,\epsilon)}(|x|)f(x)$, $h(x):=\chi_{(\epsilon,1)}(|x|)f(x)$, $v(x):=
\chi_{(1,\frac{1}{\epsilon})}(|x|)f(x)$, $w(x):=\chi_{(\frac{1}{\epsilon},\infty)}(|x|)f(x)$.
Observe that $\|f\|_{\mathcal{M}^p_q}=\|k\|_{\mathcal{M}^p_q}=\|u\|_{\mathcal{M}^p_q}$.
Further, we can compute that
\[
\|g\|_{\mathcal{M}^p_q}=\|w\|_{\mathcal{M}^p_q}=\|f\|_{\mathcal{M}^p_q}
\]
and
\[
(1-\epsilon^{-\frac{dp}{q}+d})^{\frac{1}{p}}\|f\|_{\mathcal{M}^p_q}\le \|h\|_{\mathcal{M}^p_q}=
\|v\|_{\mathcal{M}^p_q}\le \|f\|_{\mathcal{M}^p_q}.
\]
The functions that we are looking for are $F:=\frac{f}{\|f\|_{\mathcal{M}^p_q}},\ K:=
\frac{k}{\|f\|_{\mathcal{M}^p_q}}$, and $U:=\frac{u}{\|f\|_{\mathcal{M}^p_q}}$.
\end{proof}

\medskip

\noindent{\bf Acknowledgements}. This work is supported by P3MI-ITB 2019 Program. We
thank Dr. Denny I. Hakim for useful discussion regarding the functions in the proof
of Theorem \ref{Octahedral}.

\end{document}